\documentclass[a4paper,12pt]{amsart}

\usepackage{amsmath,amsthm,amsfonts,amssymb,stmaryrd,mathrsfs,enumitem}
\usepackage{latexsym}
\usepackage{fullpage}
\usepackage{a4,color,palatino,fancyhdr}
\usepackage{setspace}
\usepackage{graphicx}%poner [dvips] para grficos
\usepackage{float}  %Este packete permite controlar la flotacion de figuras al poner [H] o [h] como opcion de \begin{figure}
\usepackage[small, bf, margin=90pt, tableposition=bottom]{caption}
\RequirePackage{amssymb}
\RequirePackage[T1]{fontenc}
\usepackage{amscd}
\usepackage{xypic}
\input{xy}
\xyoption{all}
\xyoption{poly}
\usepackage[all]{xy}
\usepackage{tikz}
\newtheorem{theorem}{Theorem}[section]
\newtheorem{proposition}{Proposition}[section]
\newtheorem{definition}{Definition}
\newtheorem{lemma}{Lemma}[section]
\newtheorem{corollary}{Corollary}

\newtheorem{remark}{Remark}

\numberwithin{equation}{section}

\def\Aut{{\mathop{\rm Aut}}}

\newcommand\Z{{\mathbb{Z}}}
\newcommand\R{{\mathbb{R}}}

\title{Crossing-changeable braids from  chromatic configuration spaces}
\author{Hao Li and Zhi L\"u}

\thanks{Partially supported by the  NSFC  grants (No. 11971112 and 11431009).}

\keywords{Crossing-changable braid, chromatic configuration space, extended fundamental group.}

\address{School of Mathematical Sciences, Fudan University, Shanghai, 200433, P. R. China. }
\email{14110840001@fudan.edu.cn}

\address{School of Mathematical Sciences, Fudan University, Shanghai, 200433, P. R. China. }
 \email{zlu@fudan.edu.cn}

\begin{document}

\maketitle
\begin{abstract}
Motivated by the work in \cite{LLL}, this paper deals with the theory of the braids from chromatic configuration spaces. This kind of braids possess the property that  some strings of each  braid may intersect together and can also be untangled, so they are quite different from the ordinary braids in the sense of Artin.  This enriches and
 extends the theory of ordinary braids.

 %And do computation on graphic braid group and pure braid group when $X=\R^2$.
\end{abstract}

\section{introduction}

Generally, the mathematical study of braids can be traced back to the seminar work of Artin in \cite{A1, A2} around the first half of the last century.  The theory of braids studies the concept of braids and the presentations of braid groups as well as various generalizations arising from various branches of the mathematics. For example, braid groups can be interpreted as the fundamental groups of certain (unordered) configuration spaces by Fox and Neuwirth \cite{FN2}.  Brieskorn~\cite{B1, B2}  extended the notion of the braid group to Artin groups or the generalized braid groups by associating to all finite Coxeter groups. So far, there have been many  interesting and strong links between the theory of braids and other various theories and areas, such as knot theory, group theory,  algebraic geometry, mathematical physics and so on (e.g., see \cite{A, Ar, BCW, Bir, BS, D, J}).

\vskip .2cm

As pointed out in \cite{LLL}, the theory of ordinary braids has two basic theoretical features: One is that each braid group is realized as the fundamental group of the orbit space of a geometric object with free action of a group; the other is that each braid group uniquely corresponds to a short exact sequence induced by the geometric object with free action.
In~\cite{LLL}, the theory of orbit braids has been established by making use of the construction of orbit configuration spaces (which can be understood as a generalization of classical configuration spaces), and generally it does not possess the above basic theoretical features. The theory of the orbit braids from orbit configuration spaces provides us with much more insights. Indeed, an orbit braid group can be large enough to contain various different braid groups as subgroups, but it can still be described in terms of homotopy (i.e.,  the so-called extended fundamental group defined in \cite{LLL}).

\vskip.2cm

Motivated by the work in \cite{LLL}, in this paper we consider the braids from chromatic configuration spaces (which can be regarded as another generalization of classical configuration spaces). This kind of braids  are quite different from the ordinary braids in the sense of Artin. Actually some strings of such a braid may intersect together but we can  untangle them by the construction of chromatic configuration spaces. Of course, the groups formed by such braids don't possess the above basic theoretical features of the theory of ordinary braids yet. The purpose of this paper is to deal with the theory of such braids from chromatic configuration spaces.

\vskip .2cm
Let $\Gamma$ be a simple graph without loops on the vertex set $[n]=\{1, ..., n\}$. By $E(\Gamma)$ we denote the set of edges of $\Gamma$, and  by $\overline{ij}$ we denote the edge adjacent to two vertices $i,j\in [n]$. Then  the \textit{chromatic configuration space} of a topological space $X$ over  $\Gamma$ is defined by
$$F(X, \Gamma)=\{(x_1,\dots,x_n)\in X^{\times n}| x_j\neq x_k  \hskip.1cm \text{for} \hskip.1cm \overline{jk}\in E(\Gamma) \}\footnote{If $X$ admits an effective action of a group $G$, then we can even define the \textit{chromatic orbit configuration space} $$F_G(X, \Gamma)=\{(x_1,\dots,x_n)\in X^{\times n}| G(x_j)\neq G(x_k)  \hskip.1cm \text{for} \hskip.1cm \overline{jk}\in E(\Gamma) \}$$ where $G(x)$ denotes the orbit at $x\in X$. Certainly we can also consider the braids from the chromatic orbit configuration spaces, but  more details will be involved. Here for our purpose we mainly  pay our attention on dealing with the theory of the braids from chromatic configuration spaces.}.$$
If we regard each point in $X$ as a color, then $F(X, \Gamma)$ exactly consists of all colorings by using all colors of $X$ to color vertices of $\Gamma$ in such a way that adjacent vertices always have different colors. So this is reason why we call $F(X, \Gamma)$ the chromatic configuration space of  $X$ over  $\Gamma$.  The definition of $F(X, \Gamma)$ first appeared in the work of  Eastwood and Huggett~\cite{EH}, where $F(X, \Gamma)$ was called the generalized configuration space therein. Clearly, if $\Gamma$ is a complete graph, then $F(X, \Gamma)$ is just the classical configuration space $F(X, n)$, so the chromatic configuration spaces enrich the world of classical configuration spaces, and in particular,  they also bring in much information of graphs to configuration spaces.

\vskip .2cm

The chromatic configuration space $F(X, \Gamma)$ possesses the distinctive properties in its own way.
First, the symmetric group $\Sigma_n$ does not act on $F(X,\Gamma)$ very well, but its role will be replaced by $\text{Aut}(\Gamma)$, the automorphism group of the graph $\Gamma$. However, the canonical action of $\Aut(\Gamma)$ on $F(X,\Gamma)$ is not free in general. This results in the failure of existence of fibre map, thus it is not surprising that the braid group from the chromatic configuration space that we will consider can not be realized as the fundamental group of some topological space. Second, $F(X,\Gamma)$ relies on the structure of $\Gamma$ heavily. It is well-known that the structure of graphs is complicated, especially, the exact status of $\Aut(\Gamma)$ is still unknown \cite{BW}. These two factors lead to difficulties when considering $F(X,\Gamma)$.

\vskip.2cm
Following the idea in \cite{LLL}, the concept of braid groups can actually be generalised to various configuration spaces.
Now let us deal with the theory of the braids from the chromatic configuration spaces. Assume that $X$ is a connected topological manifold of dimension at least two.
Choose a base point ${\bf x}$ in $F(X, \Gamma)$, which is of free orbit type under the action of $\Aut(\Gamma)$.
Then we will perform our work as follows:
 \begin{enumerate}
 \item Use the paths $\alpha$ with starting point ${\bf x}$ and ending point in the orbit $\Aut(\Gamma)({\bf x})$ at ${\bf x}$ in $F(X,\Gamma)$ to construct the geometric braids $c(\alpha)$ with $n$ strings in $X\times I$.  Such a geometric braid $c(\alpha)$ is quite different from one in the sense of Artin. Actually,
     there may be intersection points between different strings in $c(\alpha)$. Intersection or non-intersection of different strings in $c(\alpha)$ depend upon the structure of $\Gamma$. However, any intersection point appearing in different strings can be untangled by doing a small deformation on the path $\alpha$ in $F(X, \Gamma)$. This means that if two strings can intersect, then any crossing happening between them can be changed via the intersection point.
      Thus,  such a geometric braid $c(\alpha)$ is also called  a \textit{crossing-changeable  braid}.
 \item Define an equivalence relation among crossing-changeable braids $c(\alpha)$ such that  the relation agrees with the homotopy relation (relative to $\partial I$) among the paths $\alpha$. Furthermore, all equivalence classes of crossing-changeable braids
     form a group, denoted by $B(X, \Gamma)$, which is called the \textit{crossing-changeable braid group}. At the same time,
     all homotopy classes (relative to $\partial I$) of the paths with starting point ${\bf x}$ and ending point in $\Aut(\Gamma)({\bf x})$ in $F(X,\Gamma)$ can also form a group by equipping with an operation, denoted by $\pi_1^E(F(X, \Gamma), {\bf x}, \Aut(\Gamma)({\bf x}))$, which is called the \textit{extended fundamental group}\footnote{As Golasi\'nski Marek told us after this paper was submitted, the extended fundamental group is actually the fundamental group  of a transformation group in the sense of Rhodes in \cite{Rh}. In addition, as noted by Looijenga in~\cite{L}, such group is  also called an equivariant fundamental group, which can be regarded as an orbifold fundamental group.
     To be compatible with \cite{LLL}, here we still adopt the notation used in \cite{LLL}.} of $F(X, \Gamma)$, defined in~\cite{LLL}.
\end{enumerate}

The following theorem tells us that the crossing-changeable braid group $B(X, \Gamma)$ can be  described in terms of the extended fundamental group $\pi_1^E(F(X, \Gamma), {\bf x}, \Aut(\Gamma)({\bf x}))$ homotopically.

\begin{theorem}
The crossing-changeable braid group $B(X,\Gamma)$ is isomorphic to the extended fundamental group $\pi_1^E(F(X,\Gamma),{\bf x}, \Aut(\Gamma)({\bf x}))$.
\end{theorem}
Compared with ordinary braid groups, the structure of $B(X,\Gamma)$ seems to be  more complicated.
As is known,  a braid in the ordinary braid group $B(X, n)$ can always be decomposed at each crossing of $n$ strings, so what we need to do is to  consider the braids with only one crossing and  then to find out the relations among them. However,  for a crossing-changeable braid in $B(X, \Gamma)$, doing the same decomposition does not work very well. In fact, since once we decompose it at each crossing, the induced braids may not lie in $B(X,\Gamma)$. This is the key point where difficulty lies in.
\vskip.2cm

Like the theory of ordinary braid groups, $B(X,\Gamma)$ contains a subgroup $P(X,\Gamma)$, called the \textit{ pure crossing-changeable braid group}, which is exactly isomorphic to the fundamental group $\pi_1(F(X, \Gamma), {\bf x})$. On the other hand, each class of $B(X,\Gamma)$ determines a unique element in $\Aut(\Gamma)$. This leads us to obtain a short exact sequence.

\begin{theorem}
There is a short exact sequence
$$1\longrightarrow P(X,\Gamma)\longrightarrow B(X,\Gamma)\longrightarrow \Aut(\Gamma)\longrightarrow 1 $$
\end{theorem}

Using the inclusion $i: F(X, n)\hookrightarrow F(X, \Gamma)$, we also study the relation between the braid groups or the extended fundamental groups of $F(X, n)$ and $F(X, \Gamma)$, see Proposition~\ref{extended exact}.

\vskip .2cm
Finally we focus on the case of $X=\mathbb{C}$. By the presentation of the classical pure braid group, we  can give a presentation of
$P(\mathbb{C},\Gamma)$ for any finite simple graph $\Gamma$ (see Theorem \ref{pure  braid group}). However, the determination of
$B(\mathbb{C}, \Gamma)$ is not an easy task since the group structure of $\Aut(\Gamma)$ is still unknown in general.
Making use of the group extension theory, we obtain an explicit presentation of $B(\mathbb{C}, C_n)$ where $C_n$ is the cycle graph.

\vskip.2cm
The paper is organized as follows. Section 2 is the main part of this paper, where we will discuss how to define a crossing-changeable braid and the equivalence relation among crossing-changeable braids. Then we construct the crossing-changeable braid group and give it a homotopy description. Furthermore, we establish a short exact sequence about $B(X,\Gamma)$ and $P(X,\Gamma)$. We also discuss the connection between the braid groups and the extended fundamental groups of $F(X, n)$ and $F(X, \Gamma)$. In Section 3, we pay our attention to the case $X=\mathbb{C}$. We calculate $P(\mathbb{C},\Gamma)$ for arbitrary finite simple graph $\Gamma$, and $B(\mathbb{C},C_n)$.

\section{Crossing-changeable braids }

Given a topological space $X$ and a simple graph $\Gamma$ without loops  on the vertex set $[n]=\{1, ..., n\}$.
 Then we have the chromatic configuration space of $X$ over  $\Gamma$  defined by
$$F(X, \Gamma)=\{(x_1,\dots,x_n)\in X^{\times n}| x_j\neq x_k  \quad \text{for} \quad \overline{jk}\in E(\Gamma) \}$$
with subspace topology. In the case when $\Gamma$ is a complete graph, then $F(X, \Gamma)$ is the classical configuration space.
However, $F(X,n)$ is always a subspace of $F(X, \Gamma)$ in general.
\vskip .2cm

Let $\Aut(\Gamma)$ denote the automorphism group of  $\Gamma$
$$\Aut(\Gamma)=\{\sigma\in\Sigma_n| \overline{\sigma(j)\sigma(k)}\in E(\Gamma) \iff \overline{jk}\in E(\Gamma) \}$$
which is a subgroup of the symmetric group $\Sigma_n$.
It is easy to see that
 the chromatic configuration space $F(X, \Gamma)$ admits a natural action of $\Aut(\Gamma)$, defined by
 $$(\sigma, (x_1, ..., x_n))\longmapsto (x_{\sigma(1)}, ..., x_{\sigma(n)}).$$
However, this  action on $\Aut(\Gamma)$ on $F(X,\Gamma)$ is generally non-free. Of course, when $\Gamma$ is a complete graph, $\Aut(\Gamma)$ just becomes the symmetric group $\Sigma_n$, so that the natural action of $\Sigma_n$ on $F(X,n)$ is free.
\vskip.5cm

In the following, we shall pay attentions to the case in which $X$ is a connected topological manifold of dimension greater than one. Then it is easy to see that $F(X, \Gamma)$ is connected.
\vskip.2cm

\subsection{Definition and equivalence relation of crossing-changeable braids}
Given a path $\alpha=(\alpha_1,\dots,\alpha_n): I \longrightarrow F(X,\Gamma)$,  where $\alpha_i: I\longrightarrow X$ is the $i$-th coordinate of $\alpha$. Then $\alpha$ uniquely determines a configuration $c(\alpha)=\{c(\alpha_1),\dots,c(\alpha_n) \}$ of $n$ strings in $X\times I$, where $c(\alpha_i)=\{(\alpha_i(s),s)|s\in I \}$. By the construction of $F(X, \Gamma)$, it is easy to see that two strings $c(\alpha_j)$ and $c(\alpha_k)$ may intersect if $\overline{jk}\not\in E(\Gamma)$. In addition, it should also be emphasized that all $n$ strings in $c(\alpha)$ are unordered in $X\times I$.
\vskip .2cm
 Now let us choose a base point ${\bf x}=(x_1, ..., x_n)$ in $F(X, \Gamma)$ such that $x_i$'s are pairwisely distinct in $X$, so
 ${\bf x}$ is of free orbit type under the action of $\Aut(\Gamma)$. Given a $\sigma\in \Aut(\Gamma)$, we denote $(x_{\sigma(1)}, ..., x_{\sigma(n)})$ by ${\bf x}_\sigma$, and denote $(\alpha_{\sigma(1)},\dots,\alpha_{\sigma(n)})$ by $\alpha_\sigma$. Then we give the following definition in the sense of Artin~\cite{A1, A2}.

\begin{definition}\label{graphic braid}
Let $\alpha=(\alpha_1,\dots,\alpha_n):I\longrightarrow F(X,\Gamma)$ be a path such that $\alpha(0)={\bf x}$ and $\alpha(1)={\bf x}_\sigma$ for some $\sigma\in \Aut(\Gamma)$. Then $c(\alpha)$ is called a \textit{crossing-changeable braid} in $X\times I$.
\end{definition}

\begin{remark}\label{geometric braid}
The crossing-changeable braids are quite different from ordinary braids in the sense of Artin. Indeed, there may be intersection points among all different strings in a crossing-changeable braid. In particular, some pairs of strings can intersect, and some pairs of strings cannot  intersect strictly. This heavily depends upon the structure of $\Gamma$.
\vskip .2cm

For example, let $\Gamma$ be the cycle graph $C_4$, then $\Aut(\Gamma)$ is the dihedral group $D_{8}=\langle a,b| a^4=b^2=e, b^{-1}ab=b^{-1}  \rangle$,  where $a$ denotes the cyclic permutation $(1234)$, and $b$ denotes the permutation $(24)$.
	\vskip.1cm
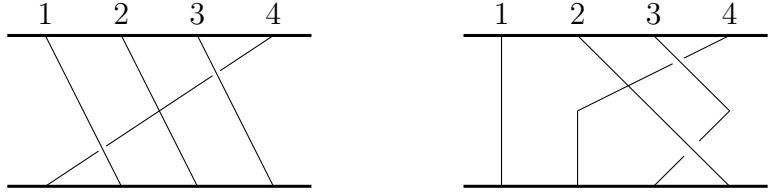
\begin{center}
	\begin{tikzpicture}

	\draw [line width=0.04cm](-5,0)--(-1,0);
	\draw [line width=0.04cm](-5,-2)--(-1,-2);
	\node[above] at (-4.5,0) {$1$};
	\node[above] at (-3.5,0) {$2$};
	\node[above] at (-2.5,0) {$3$};
	\node[above] at (-1.5,0) {$4$};
	\draw (-4.5,0)--(-3.5,-2);
	\draw (-3.5,0)--(-2.5,-2);
	\draw (-2.5,0)--(-1.5,-2);
	\draw (-1.5,0)--(-2.2,-0.467);
    \draw (-2.3,-0.533)--(-3.7,-1.467);
    \draw (-3.8,-1.533)--(-4.5,-2);
	
	\draw[line width=0.04cm] (1,0)--(5,0);
	\draw[line width=0.04cm] (1,-2)--(5,-2);
	\node[above] at (1.5,0) {$1$};
	\node[above] at (2.5,0) {$2$};
	\node[above] at (3.5,0) {$3$};
	\node[above] at (4.5,0) {$4$};
	\draw (1.5,0)--(1.5,-2);
	\draw (2.5,0)--(4.5,-2);
	\draw (3.5,0)--(4.5,-1); \draw (4.5,-1)--(4.1,-1.4); \draw (3.9,-1.6)--(3.5,-2);
	\draw (4.5,0)--(3.9,-0.3); \draw (3.75,-0.375)--(2.5,-1); \draw (2.5,-1)--(2.5,-2);
	\end{tikzpicture}
\end{center}	
\vskip .1cm
	\noindent For two braids as shown in the figures above, we notice that the second and fourth strings can intersect since  $\overline{24}\notin E(\Gamma)$. This is different from classical braids.
However, we see that  the first and fourth strings (the third and fourth strings respectively)
can not intersect since $\overline{14}, \overline{34}\in E(\Gamma)$. This just agrees with classical braids.
\vskip .2cm

By the construction of chromatic configuration spaces, obviously we can untie the intersection points appearing in the second and fourth strings of  two braids as shown in the figures above by doing  small deformations. Essentially nothing is changed up to homotopy.
\end{remark}

The observation in Remark~\ref{geometric braid} gives us an insight to the equivalence of crossing-changeable braids.
Set
$$P(F(X, \Gamma), {\bf x}, \Aut(\Gamma)({\bf x}))=\{\alpha: I\longrightarrow F(X, \Gamma)\big| \alpha(0)={\bf x}, \text{ and } \alpha(1)\in  \Aut(\Gamma)({\bf x})\}$$
which consists of those paths with  restricted endpoints  in $F(X, \Gamma)$.

\begin{definition}\label{equivalent}
Let $\alpha$ and  $\beta$ be two paths in $P(F(X, \Gamma), {\bf x}, \Aut(\Gamma)({\bf x}))$. We say that two corresponding crossing-changeable braids $c(\alpha)$ and $c(\beta)$ are \textit{ equivalent}, denoted by $c(\alpha)\sim c(\beta)$, if there exist $n$ homotopy maps $\widehat{h_i}: I\times I\longrightarrow X\times I$ given by $\widehat{h_i}(s,t)=(h_i(s,t),s)$, $i=1,\dots,n$, such that
\begin{enumerate}
\item[(1)] $\widehat{h_i}(s,0)=c(\alpha_i)$ and $\widehat{h_i}(s,1)=c(\beta_i)$;
\item[(2)] ${h_i}(0,t)=\alpha_i(0)=\beta_i(0)$ and ${h_i}(1,t)=\alpha_i(1)=\beta_i(1)$;
\item[(3)] For any $(s,t)\in I\times I$, if $\overline{jk}\in E(\Gamma)$ then $h_j(s,t)\neq h_k(s,t)$.
\end{enumerate}
\end{definition}

Actually, the  equivalence  of the crossing-changeable braids $c(\alpha)$ and $c(\beta)$ can be detected by  the homotopy equivalence of the corresponding paths $\alpha$ and  $\beta$.

\begin{proposition}\label{translation}
Let $\alpha$ and  $\beta$ be two paths in $P(F(X, \Gamma), {\bf x}, \Aut(\Gamma)({\bf x}))$. Then $\alpha\simeq \beta \hskip.1cm\text{rel} \hskip.1cm\partial I$ if and only if $c(\alpha)\sim c(\beta)$.
\end{proposition}

\begin{proof}
Assume that $h=(h_1, ..., h_n): I\times I\longrightarrow F(X,\Gamma)$ is a homotopy relative to $\partial I$ from $\alpha$ to $\beta$. Then we can use $h$ to define $n$ homotopy maps
$$\widehat{h_i}:I\times I\longrightarrow X\times I$$
by $\widehat{h_i}(s,t)=(h_i(s,t),s)$, satisfying the three conditions of Definition~\ref{equivalent}. Thus $c(\alpha)$ and $c(\beta)$ are equivalent.

\vskip .2cm
Conversely, suppose that $c(\alpha)$ and $c(\beta)$ are equivalent. Then there are $n$ homotopy maps
$$\widehat{h_i}:I\times I\longrightarrow X\times I$$
by $\widehat{h_i}(s,t)=(h_i(s,t),s)$, which satisfy the three conditions of Definition~\ref{equivalent}. These $h_i$'s determine a map $h=(h_1, ..., h_n): I\times I\longrightarrow F(X,\Gamma)$, which is just the homotopy relative to $\partial I$ from $\alpha$ to $\beta$.
\end{proof}

\subsection{Crossing-changeable braid groups and extended fundamental groups}

The concept of extended fundamental groups was given in a general way in \cite[Section 4]{LLL}, and it plays an important role on the study of orbit braids.

\vskip .2cm
By $\pi_1^E(F(X,\Gamma),{\bf x}, \Aut(\Gamma)({\bf x}))$ we denote the extended fundamental group of $F(X, \Gamma)$ with action of group $\Aut(\Gamma)$ at the base point ${\bf x}$,
 which consists of the homotopy classes (relative to $\partial I$) of all paths  in $P(F(X, \Gamma), {\bf x}, \Aut(\Gamma)({\bf x}))$, with the operation $\bullet$ defined by
 $$[\alpha]\bullet[\beta]=[\alpha\circ \beta_\sigma]$$
  for two paths $\alpha$ and $\beta$ with $\alpha(1)={\bf x}_\sigma$ and $\beta(1)={\bf x}_\tau$ in $P(F(X, \Gamma), {\bf x}, \Aut(\Gamma)({\bf x}))$, where
  $\circ$ is the usual operation between paths. Of course,  generally $[\alpha]\bullet[\beta]\not=[\beta]\bullet[\alpha]$.
  Clearly, the fundament group $\pi_1(F(X,\Gamma),{\bf x})$ is a subgroup of $\pi_1^E(F(X,\Gamma),{\bf x}, \Aut(\Gamma)({\bf x}))$. Note that all paths of $P(F(X, \Gamma), {\bf x}, \Aut(\Gamma)({\bf x}))$ are not necessarily closed.

    \vskip .2cm

 As mentioned before, all strings of each crossing-changeable braid $c(\alpha)$ at the base point ${\bf x}$ are unordered in $X\times I$, where $\alpha \in P(F(X, \Gamma), {\bf x}, \Aut(\Gamma)({\bf x}))$. This means that for any $\sigma\in \Aut(\Gamma)$, actually $c(\alpha)=c(\alpha_\sigma)$  although generally $\alpha(0)\not=\alpha_\sigma(0)$
 and  $\alpha(1)\not=\alpha_\sigma(1)$. We note that the endpoints $\alpha_\sigma(0)$
 and  $\alpha_\sigma(1)$ of $\alpha_\sigma$ are still in the orbit $\Aut(\Gamma)({\bf x})$. Thus, the set $$\{c(\alpha)|\alpha\in P(F(X, \Gamma), {\bf x}, \Aut(\Gamma)({\bf x}))\}$$
denoted by $\mathcal{C}(X\times I, \Aut(\Gamma)({\bf x}))$ consists of all possible crossing-changeable braids with endpoints lying in  $\Aut(\Gamma)({\bf x})$.

\vskip .2cm
Geometrically, there is a natural operation on all crossing-changeable braids in $\mathcal{C}(X\times I, \Aut(\Gamma)({\bf x}))$ by gluing the ending point of a crossing-changeable braid and the starting point of another crossing-changeable braid. This operation can also be defined in terms of the homotopy of paths. We state it as follows:
$$c(\alpha)*c(\beta)=c(\alpha\circ \beta_\sigma)$$
for two crossing-changeable braids $c(\alpha)$ and $c(\beta)$ with $\alpha(1)={\bf x}_\sigma$ and $\beta(1)={\bf x}_\tau$. %However, as we have known in the homotopy theory of paths, this operation $*$ is not associative.

\vskip .2cm
Let $B(X, \Gamma)$ denote the set formed by the equivalence classes of all crossing-changeable braids in $\mathcal{C}(X\times I, \Aut(\Gamma)({\bf x}))$. Then the operation $*$ on $\mathcal{C}(X\times I, \Aut(\Gamma)({\bf x}))$ induces the operation $\star$ on $B(X, \Gamma)$ as follows:
$$[c(\alpha)]\star[c(\beta)]=[c(\alpha)*c(\beta)]=[c(\alpha\circ \beta_\sigma)]$$
for two paths $\alpha$ and $\beta$ with $\alpha(1)={\bf x}_\sigma$ and $\beta(1)={\bf x}_\tau$ in $P(F(X, \Gamma), {\bf x}, \Aut(\Gamma)({\bf x}))$.

\vskip .2cm

It is not difficult to see that $B(X, \Gamma)$ forms a group under the operation $\star$, and generally it is not abelian.

\begin{definition}\label{graphic braid group}
 The group $B(X,\Gamma)$ is called the \textit{crossing-changeable braid group} of $F(X, \Gamma)$.
Those classes $[c(\alpha)]$ with $\alpha(1)={\bf x}$ in $B(X,\Gamma)$ form a subgroup, which is called the \textit{pure crossing-changeable braid group},  denoted  by $P(X,\Gamma)$.
\end{definition}

There is a homotopy description for these two  braid groups.

\begin{theorem}\label{homotopy description}\
\begin{enumerate}
\item[(1)]
The crossing-changeable braid group $B(X,\Gamma)$ is isomorphic to the extended fundamental group $\pi_1^E(F(X,\Gamma),{\bf x}, \Aut(\Gamma)({\bf x}))$;
\item[(2)]
The pure crossing-changeable braid group $P(X, \Gamma)$ is isomorphic to the ordinary fundamental group $\pi_1(F(X,\Gamma),{\bf x})$.
\end{enumerate}
\end{theorem}
\begin{proof}
 Define the map $\varphi: \pi_1^E(F(X,\Gamma),{\bf x}, \Aut(\Gamma)({\bf x}))\longrightarrow B(X,\Gamma)$ given by $$\varphi([\alpha])=[c(\alpha)].$$ For two paths $\alpha$ and $\beta$ with $\alpha(1)={\bf x}_\sigma$ and $\beta(1)={\bf x}_\tau$, a direct check shows that
   $$\varphi([\alpha]\bullet[\beta])=\varphi([\alpha\circ\beta_\sigma])=[c(\alpha\circ\beta_\sigma)]
   [c(\alpha)*c(\beta)]= [c(\alpha)]\star[c(\beta)]=\varphi([\alpha])\star\varphi([\beta])$$
    so $\varphi$ is a group homomorphism. Furthermore it follows from  Proposition~\ref{translation} that $\varphi$ is an isomorphism, as desired.

    \vskip .2cm

The  isomorphism in (2) is just the restriction of $\varphi$ to $\pi_1(F(X,\Gamma),{\bf x})$.
\end{proof}

\begin{lemma}\label{unique}
Each $[c(\alpha)]$ in $B(X,\Gamma)$ determines a unique element $\sigma\in Aut(\Gamma)$, where $\alpha(1)={\bf x}_\sigma$.
\end{lemma}
\begin{proof}
This is a direct consequence of Proposition~\ref{translation}.
\end{proof}

\begin{theorem}\label{sequence}
There is the following natural exact sequence:
$$1\longrightarrow P(X,\Gamma)\longrightarrow B(X,\Gamma)\longrightarrow \Aut(\Gamma)\longrightarrow 1$$
\end{theorem}
\begin{proof}
Take a class $[c(\alpha)]$ in $B(X,\Gamma)$, by Lemma~\ref{unique}, there is a unique element $\sigma\in \Aut(\Gamma)$ such that
 $\alpha(1)={\bf x}_\sigma$.
Define the map $\Phi: B(X,\Gamma)\longrightarrow \Aut(\Gamma)$ given by
$$\Phi([c(\alpha)])=\sigma.$$
For two paths $\alpha$ and $\beta$ with $\alpha(1)={\bf x}_\sigma$ and $\beta(1)={\bf x}_\tau$ in $P(F(X, \Gamma), {\bf x}, \Aut(\Gamma)({\bf x}))$, it is easy to see that $\alpha\circ \beta_\sigma(1)={\bf x}_{\sigma\tau}$, so
$\Phi([c(\alpha)]\star[c(\beta)])=\Phi([c(\alpha\circ\beta_\sigma)]=\sigma\tau$. Thus $\Phi$ is a group homomorphism.

\vskip .2cm

Choose an element $\sigma$ in $\Aut(\Gamma)$. Since $F(X, \Gamma)$ is connected, there must be a path $\alpha:I\longrightarrow
F(X, \Gamma)$ such that $\alpha(0)={\bf x}$ and $\alpha(1)={\bf x}_\sigma$, so $[c(\alpha)]\in B(X, \Gamma)$.
This means that $\Phi$ is surjective.

\vskip .2cm

For $[c(\alpha)]\in B(X, \Gamma)$, if $\Phi([c(\alpha)])$ is the identity of $\Aut(\Gamma)$, then $\alpha(1)={\bf x}$, so $\ker(\Phi)$ is isomorphic to $P(X,\Gamma)$. This completes the proof.
\end{proof}

\subsection{Associated with classical configuration space $F(X,n)$}

As a subspace of $F(X, \Gamma)$, the classical configuration space $F(X, n)$ admits a free action of the symmetric group $\Sigma_n$. Thus,  for any non-trivial subgroup $G$ of  $\Sigma_n$, $G$ can still  act on $F(X,n)$ freely. By \cite[Section 4]{LLL} we may define the extended fundamental group at the base point ${\bf x}$:
$$\pi_1^E(F(X,n), {\bf x}, G({\bf x}))=\{[\alpha]|\alpha:I\longrightarrow F(X,n) \text{ with } \alpha(0)={\bf x} \text{ and }
\alpha(1)\in G({\bf x})\}$$
with the operation given by
$$[\alpha]\bullet[\beta] =[\alpha\circ \beta_\sigma]$$
where $\sigma\in G\subset\Sigma_n$ is the unique element determined by $\alpha$ such that $\alpha(1)={\bf x}_\sigma$. Then we know from~\cite{LLL} that
there is the following short exact sequence
\begin{equation}
1\longrightarrow \pi_1(F(X,n), {\bf x})\longrightarrow \pi_1^E(F(X,n), {\bf x}, G({\bf x}))\longrightarrow
G\longrightarrow 1.
\end{equation}
Since $G$ is finite and the action of $G$ on $F(X, n)$ is free, by~\cite[Section 4, (B)]{LLL} we have that the group
$\pi_1^E(F(X,n), {\bf x}, G({\bf x}))$ is isomorphic to the fundamental group $\pi_1(F(X,n)/G, \overline{\bf x})$, where
$\overline{\bf x}$ is the image of ${\bf x}$ under the projection $F(X,n)\longrightarrow F(X,n)/G$.

\vskip .2cm
On the other hand, since $G({\bf x})\subset \Sigma_n({\bf x})$, it is easy to see that
$\pi_1^E(F(X,n), {\bf x}, G({\bf x}))\cong \pi_1(F(X,n)/G, \overline{\bf x})$ is a subgroup of
$\pi_1^E(F(X,n), {\bf x}, \Sigma_n({\bf x}))\cong \pi_1(F(X,n)/\Sigma_n, \widetilde{\bf x})$. It is well-known that
$\pi_1(F(X,n)/\Sigma_n, \widetilde{\bf x})$ is regarded as the braid group $B(X, n)$ in $X\times I$. Therefore,
$\pi_1^E(F(X,n), {\bf x}, G({\bf x}))\cong \pi_1(F(X,n)/G, \overline{\bf x})$ is  a subgroup of $B(X, n)$,
also denoted by $B(X, n)|_G$. Of course, this can also be seen from the theory of covering spaces.

\vskip .2cm
Now, choose $G$  as the subgroup $\Aut(\Gamma)$ of $\Sigma_n$,  let us discuss the relation between $B(X, \Gamma)$ and $B(X, n)|_{\Aut(\Gamma)}$. First, consider the natural inclusion $$i: F(X, n)\hookrightarrow F(X, \Gamma).$$
\begin{lemma}\label{epimorphism}
The induced map  $i_*: \pi_1(F(X,n), {\bf x})\rightarrow \pi_1( F(X,\Gamma), {\bf x})$ is
 an epimorphism.
\end{lemma}
\begin{proof}
Take a path class $[\alpha] \in \pi_1( F(X,\Gamma), {\bf x})$. For any $s\in I$, if $\alpha(s)\in F(X,n)$, then $[\alpha]$ is also an element of $\pi_1(F(X,n), {\bf x})$.
Otherwise, there may be some $\alpha_j$ and $\alpha_k$ as paths from $I$ to $X$  with intersection points, where $\alpha_j$ and $\alpha_k$ are two coordinates in $\alpha=(\alpha_1, ..., \alpha_n)$. Using the same idea as  in the proof of \cite[Lemma 2.7]{LLL}, we can do a slight homotopy deformation on $\alpha$ without touching endpoints to produce a new path $\alpha'$ such that all coordinates of $\alpha'$ do not intersect each other, implying that $[\alpha'] \in \pi_1(F(X,n), {\bf x})$. This deformation makes sure that
$\alpha\simeq\alpha'$ rel $\partial I$ in $F(X,\Gamma)$. Thus $i_*([\alpha'])=[\alpha']=[\alpha]$.
\end{proof}

\begin{remark}
In the proof of Lemma~\ref{epimorphism}, when  doing various different slight homotopy deformations on the path $\alpha$ in $F(X, \Gamma)$, we can obtain many new paths. Generally these new paths may not be the same up to homotopy in $F(X, n)$,  but they are always the same up to homotopy in $F(X, \Gamma)$.
\end{remark}

Since  $\Aut(\Gamma)$  can  freely act on $F(X,n)$, it is not difficult to see that the inclusion $i: F(X, n)\hookrightarrow F(X, \Gamma)$ becomes an $\Aut(\Gamma)$-map, and it also induces the homomorphism
$$i_*: \pi_1^E(F(X,n), {\bf x}, \Aut(\Gamma)({\bf x}))\rightarrow \pi_1^E( F(X,\Gamma), {\bf x}, \Aut(\Gamma)({\bf x})).$$
%so by Theorem~\ref{homotopy description} we have the homomorphism $\Lambda: B(X,n)|_{\Aut(\Gamma)}\longrightarrow B(X, \Gamma)$.
The restriction map of this homomorphism to  $\pi_1(F(X,n), {\bf x})$ is just the epimorphism in Lemma~\ref{epimorphism}. Actually this homomorphism without any restriction is still an epimorphism.
\begin{lemma}\label{epimorphism1}
The homomorphism  $i_*: \pi_1^E(F(X,n), {\bf x}, \Aut(\Gamma)({\bf x}))\rightarrow \pi_1^E( F(X,\Gamma), {\bf x}, \Aut(\Gamma)({\bf x}))$ is surjective.
\end{lemma}
\begin{proof}
The proof is  a similar argument to one of Lemma~\ref{epimorphism}.
\end{proof}

Together with Theorems~\ref{homotopy description}--\ref{sequence} and Lemmas~\ref{epimorphism}--\ref{epimorphism1}, we have

\begin{proposition}\label{extended exact}
There is an epimorphism between two short exact sequences.
\[
\begin{CD}
1@>>> \pi_1(F(X,n), {\bf x}) @>>> \pi_1^E(F(X,n), {\bf x}, \Aut(\Gamma)({\bf x})) @>>> \Aut(\Gamma) @>>> 1\\
@. @Vi_* VV   @V i_* VV  @ V= VV  \\
1@ >>> \pi_1( F(X,\Gamma), {\bf x}) @ > >> \pi_1^E( F(X,\Gamma), {\bf x}, \Aut(\Gamma)({\bf x})) @>>> \Aut(\Gamma) @>>> 1\\
\end{CD}
\]
or in terms of braid groups
\[
\begin{CD}
1@>>> P(X,n) @>>> B(X,n)|_{\Aut(\Gamma)} @>>> \Aut(\Gamma) @>>> 1\\
@. @V VV   @V VV  @ V= VV  \\
1@ >>> P(X,\Gamma) @ > >> B(X,\Gamma) @>>> \Aut(\Gamma) @>>> 1.\\
\end{CD}
\]
\end{proposition}

\section{Calculations of $B(\mathbb{C}, \Gamma)$ and $P(\mathbb{C}, \Gamma)$}

In this section we will focus on the case $X=\mathbb{C}$. We will discuss the structure of $B(\mathbb{C},\Gamma)$ and $P(\mathbb{C},\Gamma)$, we abbreviate them to $B(\Gamma)$ and $P(\Gamma)$.
When $\Gamma$ is the complete graph $K_n$ on vertex set $[n]$, $B(\Gamma)$ and $P(\Gamma)$ just are the classical braid group $B_n$ and pure braid group  $P_n$, respectively.
It is well-known that Artin~\cite{A2} and Markoff~\cite{M} gave the presentations of $B_n$ and $P_n$  respectively, which are stated as follows.
\begin{theorem}[Artin]\label{artin}
The braid group $B_n$ is generated by $\sigma_i, i=1, ..., n-1$ with the relations
 \begin{equation*}
 \begin{cases}
 \sigma_i\sigma_j=\sigma_j\sigma_i \quad\text{for}\quad |i-j|>1\\
 \sigma_i\sigma_{i+1}\sigma_i=\sigma_{i+1}\sigma_i\sigma_{i+1}
\end{cases}
\end{equation*}
where $\sigma_i$ is the  braid with only one crossing, as shown below
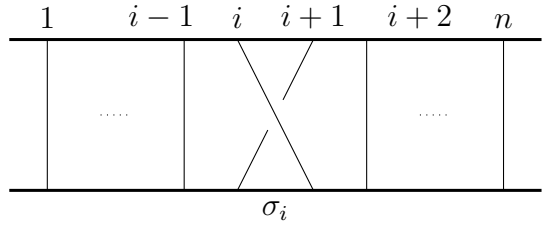
\begin{center}
\begin{tikzpicture}
\draw [line width=0.04cm] (0,0)--(7,0);
%\node[right] at (7,0) {$t=0$};
\draw [line width=0.04cm] (0,-2)--(7,-2);
%\node[right] at (7,-2) {$t=1$};

\draw (0.5,0)--(0.5,-2);
\node[above] at (0.5,0) {$1$};

\draw[dotted] (1.2,-1)--(1.6,-1);

% %????
\draw (2.3,0)--(2.3,-2);
\node[above] at (2,0) {$i-1$};

\draw (3,0)--(4,-2);
\node[above] at (3,0) {$i$};

\draw (4,0)--(3.6,-0.8);
\draw (3.4,-1.2)--(3,-2);
\node[above] at (4,0) {$i+1$};

\draw (4.7,0)--(4.7,-2);
\node[above] at (5.4,0) {$i+2$};

\draw (6.5,0)--(6.5,-2);
\node[above] at (6.5,0) {$n$};

\draw[dotted] (5.4,-1)--(5.8,-1);

\node[below] at (3.5,-2) {$\sigma_i$};
\end{tikzpicture}
\end{center}
\end{theorem}

Let  $s_{i,j}$, $1\leq i < j \leq n$ be the elements  of $B_n$ given  by the formulae
\begin{equation}\label{s}
s_{i,j}=\sigma_{j-1}\dots\sigma_{i+1}\sigma_i^2\sigma_{i+1}^{-1}\dots\sigma_{j-1}^{-1}
\end{equation}
as shown below
\begin{center}
\begin{tikzpicture}
\draw [line width=0.04cm](0,0)--(8,0);
\draw [line width=0.04cm](0,-2)--(8,-2);
\node[below] at (4,-2){$s_{i,j}$};
\node[below] at (4, -2.5)  {\text{Figure A}};

\draw (0.5,0)--(0.5,-2);
\node[above] at (0.5,0){$1$};

\draw[dotted] (1,-1)--(1.3,-1);
\draw (2,0)--(2,-1.0);
\draw (2,-1.2)--(2,-2);
\node[above] at (2,0){$i$};

\draw (3,0)--(3,-2);
\node[above] at (3,0){$i+1$};

\draw[dotted] (3.8,-1)--(4.2,-1);
\draw (5,0)--(5,-2);
\node[above] at (5,0){$j-1$};

\draw (6,0)--(5.1,-0.2);
\draw (4.9,-0.244)--(4.5,-0.333);
\draw (3.5,-0.556)--(3.1,-0.644);
\draw (2.9,-0.689)--(2.1,-0.867);
\draw (1.9,-0.911)--(1.5,-1);

\draw (1.5,-1)--(2.9,-1.311);
\draw (6,-2)--(5.1,-1.8);
\draw (4.9,-1.756)--(4.5,-1.667);
\draw (3.5,-1.444)--(3.1,-1.356);
\node[above] at (6,0){$j$};

\draw[dotted] (6.6,-1)--(6.9,-1);

\draw(7.5,0)--(7.5,-2);
\node[above] at (7.5,0){$n$};
\end{tikzpicture}
\end{center}

\begin{theorem}[Markoff]\label{markoff}
The elements $s_{i,j}$, $1\leq i < j \leq n$ with the Burau relations:
\begin{enumerate}
\item[$(1)$] $s_{i,j}s_{k,l}=s_{k,l}s_{i,j}$\hskip4.5cm for $i<j<k<l$ and for $i<k<l<j$,
\item[$(2)$] $s_{i,j}s_{i,k}s_{j,k}=s_{i,k}s_{j,k}s_{i,j}=s_{j,k}s_{i,j}s_{i,k}$\hskip1.1cm for $i<j<k$,
\item[$(3)$] $s_{i,k}s_{j,k}s_{j,l}s_{j,k}^{-1}=s_{j,k}s_{j,l}s_{j,k}^{-1}s_{i,k}$ \hskip2.1cm for $i<j<k<l$.
\end{enumerate}
 give a presentation of pure braid group $P_n$.
\end{theorem}

\begin{remark}Intuitively we  see that each generator in $B_n$ and $P_n$ corresponds to an edge in $K_n$. In particular, all generators of $P_n$ bijectively correspond to all edges of $K_n$.
\end{remark}

\subsection{The presentation of $P(\Gamma)$} Let $\Gamma$ be a simple graph without loops on vertex set $[n]$. Now  let us discuss the the presentation of $P(\Gamma)$, which is the generalization of $P_n$.
 \vskip .2cm
   Define $s_{ij}$ to be the braid in $P(\Gamma)$ with  the $i$-th string and the $j$-th string tangled one time,  as shown in Figure A.
Obviously, if $\overline{ij}\not\in E(\Gamma)$, then $s_{ij}$ will be trivial. Let $\Delta_{ijk}$ denote the 3-circuit in $K_n$, which is formed by three edges $\overline{ij}, \overline{ik}, \overline{jk}$ where $i,j,k\in [n]$. Let
$\square_{ijkl}$ denote the 4-circuit in $K_n$, which is formed by four edges $\overline{ij}, \overline{jk}, \overline{kl}, \overline{il}$ where $i,j,k, l\in [n]$.

\begin{theorem}\label{pure braid group}
 The elements $s_{i,j}$, $\overline{ij}\in E(\Gamma)$ with the following relations
\begin{enumerate}
\item[$(1)$]  if $\overline{ij}, \overline{kl}\in E(\square_{ijkl})\cap E(\Gamma)$ with $i<j<k<l$ or $i<k<l<j$, then
$$s_{i,j}s_{k,l}=s_{k,l}s_{i,j};$$
\item[$(2_1)$] if $\Delta_{ijk}$ with $i<j<k$ is a 3-circuit in $\Gamma$, then $$s_{i,j}s_{i,k}s_{j,k}=s_{i,k}s_{j,k}s_{i,j}=s_{j,k}s_{i,j}s_{i,k};$$
\item[$(2_2)$]  if $\Delta_{ijk}$ is not a 3-circuit in $\Gamma$ but $\overline{ij}, \overline{jk}\in E(\Gamma)$ (so  $\overline{ik}\notin E(\Gamma)$), then
$$s_{i,j}s_{j,k}=s_{j,k}s_{i,j}$$
where the restriction condition $i<j<k$ is not necessarily satisfied;
\item[$(3_1)$] if $\overline{ik}, \overline{jl}\in E(\square_{ijkl})\cap E(\Gamma)$ with  $i<j<k<l$ and $\Delta_{jkl}$ is a 3-circuit in $\Gamma$, then
$$s_{i,k}s_{j,k}s_{j,l}s_{j,k}^{-1}=s_{j,k}s_{j,l}s_{j,k}^{-1}s_{i,k};$$
\vskip.2cm
\item[$(3_2)$]  if $\overline{ik}, \overline{jl}\in E(\square_{ijkl})\cap E(\Gamma)$ with  $i<j<k<l$ but $\Delta_{jkl}$ is not a 3-circuit in $\Gamma$, then
    $$s_{i,k}s_{j,l}=s_{j,l}s_{i,k}$$
\end{enumerate}
give a presentation of the pure crossing-changable braid group $P(\Gamma)$.
\end{theorem}
\begin{proof}
We know from Proposition~\ref{extended exact} that the induced map $i_*: P_n\longrightarrow P(\Gamma)$ by the inclusion
$i: F(\mathbb{C}, n)\hookrightarrow F(\mathbb{C}, \Gamma)$ is an epimorphism. This means that all elements $s_{i,j}$, $\overline{ij}\in E(\Gamma)$, form a system of generators in $P(\Gamma)$, and all possible relations among them can be obtained by
the images of the epimorphism $i_*$ acting on the Burau relations (1)--(3) in Theorem~\ref{markoff}.
Therefore, the required relations follow easily from direct calculations.
\end{proof}

\begin{remark}
Randell tells us in \cite{R, RR} that there is also another approach to get Theorem~\ref{pure  braid group}  from the viewpoint of complex hyperplane arrangements.
  Actually $F(\mathbb{C},\Gamma)$ can be regarded as a complement space of complex hyperplane arrangements as follows: Let $H_{i,j}:=\{(z_1,\dots,z_n)\in \R^{2n}| z_i=z_j \}$. Then  $F(\mathbb{C},\Gamma)=\mathbb{C}^n\setminus \underset{\overline{ij}\in E(\Gamma)}{\bigcup} H_{i,j}$. Randell's method is  more geometric, for more details, see~\cite{R, RR}). Here the method used is of the combinatorial nature.
\end{remark}

\begin{corollary}\label{free abelian}
If $\Gamma$ doesn't contain any 3-circuit, then $P(\Gamma)$ is a free abelian group generated by $s_{i,j}$, $\overline{ij}\in E(\Gamma)$.
\end{corollary}

\subsection{The presentation of B($\Gamma$)}

The group structure of  $B(\Gamma)$ heavily relies on $\Aut(\Gamma)$. However, as pointed out in~\cite{BW}, the group structure of $\Aut(\Gamma)$ is still an unsolved problem except for some particular classes of graphs.
Furthermore, we see from the short exact sequence
$$1\rightarrow P(\Gamma) \rightarrow B(\Gamma)\rightarrow \Aut(\Gamma)\rightarrow 1$$
(see Theorem~\ref{sequence}) that  the determination of $B(\Gamma)$ would be quite difficult in general although a presentation of $P(\Gamma)$ has been given in Theorem~\ref{pure  braid group}.
\vskip.2cm
For a short exact sequence of groups $$1\rightarrow  A \xrightarrow j E\xrightarrow{p} G\rightarrow 1,$$
the group extension theory \cite{PCG} tells us how
 to give a presentation of $E$ from the given  $A=\langle S\hskip.1cm|\hskip.1cm R  \rangle$ and $G=\langle T\hskip.1cm| \hskip.1cmQ \rangle$.  Choose a mapping (not necessarily a homomorphism) $\psi: G\rightarrow E$ with $p\circ \psi=\text{id}_G$. Clearly the choice of $\psi$ is not unique since $p$ is surjective. Then we use $\psi$ to give a conjugation action $\gamma$ of $G$ on $A$ defined by $ \gamma(g, a)=\psi(g)^{-1}a\psi(g)$ where $g\in G$ and $a\in A$. For each $t\in T$, $\gamma(t, \cdot)$ defines an automorphism $\gamma_t:A\rightarrow A$. Moreover,  $E$ is presented as follows
$$\langle S\cup \psi(T)\hskip.1cm|\hskip.1cm R \cup \{\psi(t)^{-1}s\psi(t)=\gamma_t(s): s\in S, t\in T \}\cup \{\omega_q(s)=  q(\psi(t); t\in T): q\in Q \} \rangle $$
where $w_q(s)$ is a word  in $S$ such that $j(\omega_q(s))=q(\psi(t))$ in $E$, and each $q\in Q$ corresponds to a word $q(t)$ equal to the identity $e_G$ in $G$ so  $p(q(\psi(t))=e_G$.
\vskip.2cm

\subsubsection{A presentation of $B(C_n)$}
With the above understanding, now we consider the case in which  $\Gamma$ is the cycle graph $C_n$ in $K_n$. Without loss of generality, assume that $C_n$ is the cycle graph formed by edges $\overline{12}, \overline{23}, ..., \overline{n (n-1)}$ and $\overline{1n}$.
 Then $\Aut(C_n)$ is the dihedral group $$\mathcal{D}_{2n}=\langle a,b| a^n=b^2=e, bab=a^{-1}  \rangle$$
 where $a$ denotes the cyclic permutation $(12\cdots n)$ and $b$ denotes the permutation
 $$b=
 (2\ n)(3\ n-1)\cdots (r(n) \  s(n))$$
 where $r(n)=\begin{cases}
{{n}\over 2}   & \text{if $n$ is even}\\
{{n+1}\over 2} & \text{if $n$ is odd}
\end{cases}
$
and
$s(n)=\begin{cases}
{{n+4}\over 2} & \text{if $n$ is even}\\
{{n+3}\over 2} & \text{if $n$ is odd}.
\end{cases}$

\vskip .2cm
By Corollary~\ref{free abelian}, $P(C_n)$ is a free abelian group generated by $s_{i,i+1} (1\leq i\leq n-1)$ and $s_{1,n}$, so $$P(C_n)\cong \underset{1\leq i\leq n-1}{\bigoplus}\Z(s_{i,i+1})\oplus \Z(s_{1,n}).$$
Next we are going to  give a presentation of $B(C_n)$ by making use of  the group extension theory from the short exact sequence
$$1\rightarrow P(C_n) \xrightarrow{j} B(C_n)\xrightarrow{p} \mathcal{D}_{2n}\rightarrow 1.$$

\vskip .2cm
By Proposition~\ref{extended exact}, we have the following commutative diagram with three epimorphisms
$$
\xymatrix{
            B_n|_{\Aut(C_n)}              \ar[dd]^{i_*} \ar[dr]  &  \\
   & \mathcal{D}_{2n} \\
              B(C_n)\ar[ur]^p & \\
             }
$$
Choose a map $\psi: \mathcal{D}_{2n}\longrightarrow B(C_n)$ with $p\circ \psi=\text{id}_{\mathcal{D}_{2n}}$, defined by
 $\psi(a)=i_*(a_{1,n})$ and
 $$\psi(b)=
  i_*(a_{2,n}a_{3,n}^{-1}a_{3,n-1}a_{4,n-1}^{-1}\cdots a_{r(n), s(n)}a^{-1}_{r(n)+1, s(n)})  $$
where $a_{i,j}$ denotes the word $\sigma_{j-1}\cdots \sigma_i$ in $B_n$ (for the meanings of $\sigma_i$'s, see Theorem~\ref{artin}); namely
\begin{equation}\label{a}
a_{i,j}=\sigma_{j-1}\cdots \sigma_i.
\end{equation}
 Geometrically,  $a_{k,l}a_{k+1,l}^{-1}$ expresses the braid of exchanging
 the $k$-th and $l$-th strings, which is abbreviated as $e_{k,l}$ for  convenience. Then
 $$\psi(b)=
  i_*(e_{2,n}e_{3,n-1}\cdots e_{r(n),s(n)}). $$
 \begin{remark}\label{i_*}
 Since $i_*: B_n|_{\Aut(C_n)}\longrightarrow B(C_n)$ is induced by the conclusion $i: F(\mathbb{C}, n)\hookrightarrow F(\mathbb{C}, C_n)$, it is easy to see that for any braid $\rho$ in $B_n|_{\Aut(C_n)}$,
 $i_*(\rho)$ can be understood as the braid of $\rho$ in $B(C_n)$, so we may write $i_*(\rho)=\rho$.
 \end{remark}
\vskip.2cm
First let us prove some useful equations in $B_n$.
\begin{lemma}\label{equations} In  $B_n$, the following equations hold:
\begin{enumerate}
\item[$(1^\circ)$] $\sigma_{i+1}\sigma_i\sigma_{i+1}^2=\sigma_i^2\sigma_{i+1}\sigma_i$;
\item[$(2^\circ)$] $a_{i,j}\sigma_k=\sigma_{k-1}a_{i,j}$ for $i<k<j$
\item[$(3^\circ)$] $a_{i,j}a_{k,l}=a_{k-1,l-1}a_{i,j}$ for $i<k<l\leq j$;
\item[$(4^\circ)$] $e_{2,n}， e_{3,n-1}, ...,  e_{r(n),s(n)}$ are pairwise commutative;
\item[$(5^\circ)$] $e_{k,n+2-k}a_{1,n}=a_{1,n}e_{k+1,n+3-k}$ for $2< k\leq [\frac{n+1}{2}]$.
\end{enumerate}
\end{lemma}
\begin{proof} Making use of the relations (1)--(2) in Theorem~\ref{artin}, direct calculations give
$$\sigma_{i+1}\sigma_i\sigma_{i+1}^2=\sigma_{i+1}\sigma_i\sigma_{i+1}\sigma_{i+1}=\sigma_i\sigma_{i+1}\sigma_i\sigma_{i+1}=
\sigma_i\sigma_i\sigma_{i+1}\sigma_i=\sigma_i^2\sigma_{i+1}\sigma_i$$
and
\begin{align*} a_{i,j}\sigma_k =&\sigma_{j-1}\dots\sigma_k\sigma_{k-1}\dots\sigma_i\sigma_k=\sigma_{j-1}\dots\sigma_k\sigma_{k-1}\sigma_k\dots\sigma_i\\
=&\sigma_{j-1}\dots\sigma_{k-1}\sigma_k\sigma_{k-1}\dots\sigma_i=\sigma_{k-1}\sigma_{j-1}\dots\sigma_i\\
=&\sigma_{k-1}a_{i,j}.
\end{align*}
This proves $(1^\circ)$ and $(2^\circ)$.
Using the equation $(2^\circ)$, we have $$a_{i,j}a_{k,l}=a_{i,j}\sigma_{l-1}\cdots\sigma_k=\sigma_{l-2}\cdots\sigma_{k-1}a_{i,j}=a_{k-1,l-1}a_{i,j} $$
as desired in $(3^\circ)$. On $(4^\circ)$,  for arbitrary $k<l$, repeating the use of the equation $(3^\circ)$, we have
\begin{align*}
e_{k, n+2-k}e_{l,n+2-l}=&
a_{k,n+2-k}a_{k+1,n+2-k}^{-1}a_{l,n+2-l}a_{l+1,n+2-l}^{-1}\\
=&a_{k,n+2-k}a_{l+1,n+3-l}a_{k+1,n+2-k}^{-1}a_{l+1,n+2-l}^{-1}\\
=&a_{l,n+2-l}a_{k,n+2-k}a_{l+2,n+3-l}^{-1}a_{k+1,n+2-k}^{-1}\\
=&a_{l,n+2-l}a_{l+1,n+2-l}^{-1}a_{k,n+2-k}a_{k+1,n+2-k}^{-1}\\
=&e_{l,n+2-l}e_{k, n+2-k}
\end{align*}
as desired. On $(5^\circ)$,   when $2< k\leq [\frac{n+1}{2}]$, repeating the use of the equation $(3^\circ)$ gives
\begin{align*}
e_{k,n+2-k}a_{1,n}= & a_{k,n+2-k}a_{k+1,n+2-k}^{-1}a_{1,n}\\
= &a_{k,n+2-k}a_{1,n}a_{k+2,n+3-k}^{-1}=a_{1,n}a_{k+1,n+3-k}a_{k+2,n+3-k}^{-1}\\
= &a_{1,n}e_{k+1,n+3-k}.
\end{align*}
\end{proof}

\begin{proposition}
 The elements $s_{i,j}, \overline{ij}\in E(C_n)$, $\psi(a)$ and $\psi(b)$ with three families of relations:
\begin{enumerate}
\item[$R_1:$] $s_{i,j}, \overline{ij}\in E(C_n)$ are pairwise  commutative;
\vskip.2cm
\item[$R_2:$] $\psi(a)^{-1}s_{i,j}\psi(a)=\begin{cases}
s_{a(i),a(j)} & \text{if $1\leq i< n-1$ and $j=i+1$}  \\
s_{a(j), a(i)} & \text{if $i=1$ and $j=n$}\\
s_{a(j), a(i)} & \text{if $i=n-1$ and $j=n$};\\
\end{cases}$ \\
$\psi(b)^{-1}s_{i,j}\psi(b)=\begin{cases}
s_{b(j),b(i)} & \text{if $1< i\leq n-1$ and $j=i+1$}  \\
s_{b(i),b(j)} & \text{if $i=1$ and $j=2$}\\
s_{b(i),b(j)} & \text{if $i=1$ and $j=n$};\end{cases}$
\vskip.2cm
\item[$R_3:$] $\psi(a)^n=s_{1,n}s_{n-1,n}s_{n-2,n-1}\cdots s_{1,2}$;\\
\vskip .1cm
$\psi(b)^2=\begin{cases}
e & \text{ if $n$ is even}\\
s_{\frac{n+1}{2},\frac{n+3}{2}} & \text{ if $n$ is odd};
\end{cases}$\\
\vskip .1cm
$\psi(b)\psi(a)\psi(b)\psi(a)=\begin{cases}
s_{1,2}s_{\frac{n+2}{2},\frac{n+4}{2}} & \text{ if $n$ is even}\\
s_{1,2}s_{\frac{n+1}{2},\frac{n+3}{2}}s_{\frac{n+3}{2},\frac{n+5}{2}} & \text{ if $n$ is odd}
\end{cases}$
\end{enumerate}
give a presentation of  the crossing-changable braid group $B(C_n)$.
\end{proposition}

\begin{proof}
By the theory of group extension, it suffices to show that three families of relations $R_1$--$R_3$ hold.
\vskip.2cm			
The relation $R_1$ directly follows from the commutativity of $P(C_n)$.
\vskip.2cm
The relation $R_2$ comes from the conjugation action of $\mathcal{D}_{2n}$ on $P(C_n)$. We convention that $\overline{ij}\in E(C_n)$
means for $1\leq i\leq n-1, j=i+1$ and for $i=n$, $j=1$.
We proceed as follows. On $\psi(a)^{-1}s_{i,j}\psi(a)$ with  $\overline{ij}\in E(C_n)$,
\begin{align*}
\psi(a)^{-1}s_{i,j}\psi(a)=& \psi(a)^{-1}\sigma_i^2\psi(a)=i_*(a_{1,n}^{-1}\sigma_i^2a_{1,n})\\
=&
\begin{cases}
i_*(a_{1,n}^{-1}a_{1,n}\sigma_{i+1}^2) & \text{for $1\leq i\leq n-2$}\\
i_*(a_{1,n}\sigma_1^2a_{1,n}^{-1}) & \text{for $i=n-1$}\\
i_*(a_{1,n}^{-1}a_{2,n}\sigma_1^2a_{2,n}^{-1}a_{1,n}) & \text{for $i=n$}
\end{cases}
\text{ \ \ \ (by Lemma~\ref{equations}$(1^\circ)-(2^\circ)$)}\\
=& \begin{cases}
s_{i+1,i+2} & \text{for $1\leq i\leq n-2$}\\
s_{1,n} & \text{for $i=n-1$}\\
s_{1,2} & \text{for $i=n$}
\end{cases}
\text{ \ \ \ (by (\ref{s}) and (\ref{a}))}\\
=& s_{a(i),a(j)}.
\end{align*}
On $\psi(b)^{-1}s_{i,j}\psi(b)$  with  $\overline{ij}\in E(C_n)$, we  have that
\begin{align*}
\psi(b)^{-1}s_{i,j}\psi(b)
= &\begin{cases}
 e^{-1}_{\frac{n}{2},\frac{n+4}{2}}\cdots e^{-1}_{2,n}s_{i,j}e_{2,n}\cdots e_{\frac{n}{2},\frac{n+4}{2}} &
\text{ if $n$ is even}\\
e^{-1}_{\frac{n+1}{2},\frac{n+3}{2}}\cdots e^{-1}_{2,n}s_{i,j}e_{2,n}\cdots e_{\frac{n+1}{2},\frac{n+3}{2}} &
\text{ if $n$ is odd}.\\
\end{cases}
%\text{ \ \ \ (by Lemma~\ref{equations}$(4^\circ)$)}\\
\end{align*}
In particular, if $\overline{ij}\not\in E(C_n)$ then $s_{i,j}$ is a trivial braid so $\psi(b)^{-1}s_{i,j}\psi(b)=e$.
First let us look at the conjugation
 $e_{k, n+2-k}^{-1}s_{i,j}e_{k,n+2-k}$  where  $1<k\leq [\frac{n+1}{2}]$ and $\overline{ij}\in E(C_n)$. If each of $i$ and $j$ is not equal to $k$ or $n+2-k$, then we easily see from the
geometric meanings of $e_{k, n+2-k}$ and $s_{i,j}$ that $e_{k, n+2-k}$ and $s_{i,j}$ are commutative so $$e_{k, n+2-k}^{-1}s_{i,j}e_{k,n+2-k}=s_{i,j}.$$
Thus, by Lemma~\ref{equations}$(4^\circ)$, it suffices to consider
the case in which one of $i$ and $j$ is equal to $k$ or $n+2-k$. The argument proceeds as follows:
\vskip .2cm
When $i=1, j=2$,
\begin{align*}
\psi(b)^{-1}s_{1,2}\psi(b)&=i_*(e_{2,n}^{-1}s_{1,2}e_{2,n})=i_*(a_{3,n}a_{2,n}^{-1}\sigma_1^2a_{2,n}a_{3,n}^{-1})=i_*(a_{3,n}\sigma_2^{-1}\sigma_1^2\sigma_2a_{3,n}^{-1})\\
&=i_*(s_{2,n}^{-1}s_{1,n}s_{2,n})=s_{1,n}=s_{b(1),b(2)}.
\end{align*}

When $i=1, j=n$,
\begin{align*}
\psi(b)^{-1}s_{1,n}\psi(b)=i_*(e_{2,n}^{-1}s_{1,n}e_{2,n})=i_*(a_{3,n}a_{2,n}^{-1}a_{2,n}\sigma_1^2a_{2,n}^{-1}a_{2,n}a_{3,n}^{-1})
=i_*(\sigma_1^2)=s_{1,2}=s_{b(1),b(n)}.
\end{align*}

When $i=l-1, j=l$ where $2<l\leq n$, if $2<l\leq [\frac{n+1}{2}]$ then
\begin{align*}
\psi(b)^{-1}s_{l-1,l}\psi(b)=&i_*(e_{l-1,n+3-l}^{-1}e_{l,n+2-l}^{-1}s_{l-1,l}e_{l,n+2-l}e_{l-1,n+3-l})\\
=&i_*(e_{l-1,n+3-l}^{-1}a_{l+1,n+2-l}a_{l,n+2-l}^{-1}s_{l-1,l}a_{l,n+2-l}a_{l+1,n+2-l}^{-1}e_{l-1,n+3-l})\\
=&i_*(e_{l-1,n+3-l}^{-1}a_{l+1,n+2-l}\sigma_l^{-1}s_{l-1,l}\sigma_la_{l+1,n+2-l}^{-1}e_{l-1,n+3-l})\\
=&i_*(e_{l-1,n+3-l}^{-1}s_{l, n+2-l}^{-1}s_{l-1,n+2-l}s_{l, n+2-l}e_{l-1,n+3-l})\\
=&i_*(s_{l, n+2-l}^{-1}e_{l-1,n+3-l}^{-1}s_{l-1,n+2-l}e_{l-1,n+3-l}s_{l, n+2-l})\\
=&i_*(s_{l, n+2-l}^{-1}a_{l-1,n+3-l}^{-1}a_{l-1,n+2-l}a_{l-1,n+1-l}^{-1}\sigma_{n+1-l}^2a_{l-1,n+1-l}a_{l-1,n+2-l}^{-1}\\
&\cdot a_{l-1,n+3-l}s_{l, n+2-l})\\
=&i_*(s_{l, n+2-l}^{-1}a_{l-1,n+3-l}^{-1}\sigma_{n+1-l}^2a_{l-1,n+3-l}s_{l, n+2-l})\\
=&i_*(s_{l, n+2-l}^{-1}a_{l-1,n+1-l}^{-1}\sigma_{n+1-l}^{-1}\sigma_{n+1-l}\sigma_{n+2-l}^2\sigma_{n+1-l}^{-1}\sigma_{n+1-l}a_{l-1,n+1-l}s_{l, n+2-l})\\
=&i_*(s_{l, n+2-l}^{-1}s_{n+2-l, n+3-l}s_{l, n+2-l})=s_{n+2-l, n+3-l}=s_{b(l-1),b(l)};
\end{align*}
if $[\frac{n+1}{2}]< l \leq n$,  in the similar way as above, we can still obtain $\psi(b)^{-1}s_{l-1,l}\psi(b)=s_{b(l-1),b(l)}$.

\vskip.2cm
In conclusion, $\psi(b)^{-1}s_{i,j}\psi(b)=\begin{cases}
s_{b(j),b(i)} & \text{if $1< i\leq n-1$ and $j=i+1$}  \\
s_{b(i),b(j)} & \text{if $i=1$ and $j=2$}\\
s_{b(i),b(j)} & \text{if $i=1$ and $j=n$}\end{cases}$ as desired.

\vskip.2cm
Finally let us consider the relation $R_3$, which essentially comes from the relations in the dihedral group $\mathcal{D}_{2n}$.
Making use of  Reamrk~\ref{i_*}, Lemma~\ref{equations} and (\ref{s})--(\ref{a}), by direct calculations we have
\begin{align*}
\psi(a)^n= & i_*(a_{1,n}^n)=i_*((\sigma_{n-1}\sigma_{n-2}\cdots\sigma_1^2\cdots\sigma_{n-2}\sigma_{n-1})(\sigma_{n-2}\cdots
\sigma_1^2\cdots\sigma_{n-2})\cdots \sigma_1^2)\\
= & i_*((s_{1,n}s_{2,n}\cdots s_{n-1,n})(s_{1,n-1}s_{2,n-1}\cdots s_{n-2,n-1})\cdots s_{1,2})\\
=& s_{1,n}s_{n-1,n}s_{n-2,n-1}\cdots s_{1,2}
\end{align*}
and
\begin{align*}
\psi(b)^2= &
 e_{2,n}^2e_{3,n-1}^2\cdots e^2_{r(n), s(n)}
  \text{\ \ \ (by  Lemma~\ref{equations}$(4^\circ)$)}\\
= &
 s_{2,n}s_{3,n-1}\cdots s_{r(n), s(n)}
  \text{\ \ \ (by (~\ref{s}))}\\
= &
\begin{cases}
 e & \text{ if $n$ is even}\\
  s_{\frac{n+1}{2}\frac{n+3}{2}} & \text{ if $n$ is odd}
  \end{cases}
  \text{\ \ \ (since  $\overline{ij}\not\in E(C_n)$ with $|i-j|>1$ except $\overline{1n}\in E(C_n)$).}\\
\end{align*}
On $\psi(b)\psi(a)\psi(b)\psi(a)$, by direct calculations we have that when $n=4$
\begin{align*}
\psi(b)\psi(a)\psi(b)\psi(a)= &i_*((a_{2,4}a_{3,4}^{-1}a_{1,4})^2)=(a_{2,4}a_{1,3})^2=(\sigma_3\sigma_2\sigma_2\sigma_1)^2=\sigma_3^2\sigma_1^2\\
= & s_{1,2}s_{3,4} \text{ (since $\overline{24}\notin E(C_4)$)}
\end{align*}
and when $n=5$,
\begin{align*} \psi(b)\psi(a)\psi(b)\psi(a)= & i_*((a_{2,5}a_{3,5}^{-1}a_{3,4}a_{1,5})^2)=(a_{2,5}a_{3,5}^{-1}a_{1,5}a_{4,5})^2\\
=&(a_{2,5}a_{1,3}a_{4,5})^2=(\sigma_4\sigma_3\sigma_1\sigma_4)^2=\sigma_4^2\sigma_3^2\sigma_1^2\\
= & s_{1,2}s_{3,4}s_{4,5}
 \text{ ( since $\overline{35}\notin E(C_5)$)}.
\end{align*}
More generally, when $n\geq 6$, %since $\overline{2n}\notin E(C_n)$,
by Lemma~\ref{equations} we have that
\begin{align*}
\psi(b)\psi(a)\psi(b)\psi(a)= & i_*(e_{2,n}e_{3,n-1}\cdots e_{r(n), s(n)}a_{1,n})^2=i_*(e_{2,n}a_{1,n}e_{4,n}\cdots e_{r(n)+1, s(n)+1})^2\\
= & i_*(a_{2,n}a_{3,n}^{-1}a_{1,n}e_{4,n}\cdots e_{r(n)+1, s(n)+1})^2=i_*(a_{3,n}\sigma_2^2\sigma_1e_{4,n}\cdots e_{r(n)+1, s(n)+1})^2\\
= &i_*(a_{3,n}e_{4,n}\cdots e_{r(n)+1, s(n)+1})^2\sigma_1^2 =i_*(a_{3,n}e_{4,n}\cdots e_{r(n)+1, s(n)+1})^2s_{1,2}.
\end{align*}
 We see that $i_*(a_{3,n}e_{4,n}\cdots e_{r(n)+1, s(n)+1})^2$ is only revelent to the path
  $$\Gamma': 3\text{ ---} 4\text{ ---}\cdots\text{ ---}n$$ in $C_n$,
 which  can  be regarded as the path
 $1$---$2$---$\cdots$---$(n-2)$ of $C_{n-2}$ on vertex set $[n-2]=\{1, 2,..., n-2\}$. Furthermore, in a similar way to the proof of Lemma~\ref{epimorphism}, we may obtain that the inclusion $j: F(\mathbb{C}, C_{n-2})\longrightarrow F(\mathbb{C},\Gamma')$ induces an epimorhpism $j_*: B(C_{n-2})|_{\Aut(\Gamma')}\longrightarrow B(\Gamma')$. Thus  $$i_*(a_{3,n}e_{4,n}\cdots e_{r(n)+1, s(n)+1})^2$$ becomes $$j_*(i_*(a_{1,n-2}e_{2,n-2}\cdots e_{r(n-2)+1, s(n-2)+1} )^2),$$ in which
$i_*(a_{1,n-2}e_{2,n-2}\cdots e_{r(n-2)+1, s(n-2)+1} )^2$  is exactly $(\psi(a)\psi(b))^2$ in $B(C_{n-2})$. This means that we can get the expression of $\psi(b)\psi(a)\psi(b)\psi(a)$ by induction.
\vskip .2cm

We have proved the cases of $n=4, 5$. Now assume inductively that $$\psi(b)\psi(a)\psi(b)\psi(a)=
\begin{cases}
s_{1,2}s_{k+1,k+2} & \text{if $n=2k$}\\
s_{1,2}s_{k+1,k+2}s_{k+2, k+3} & \text{if $n=2k+1$}.
\end{cases}
$$
Then, as  discussed above,
  \begin{align*}
  i_*(a_{1,n}e_{2,n}\cdots e_{r(n)+1, s(n)+1})^2=
  &(\psi(a)\psi(b))^2
=\psi(b)^{-1}\psi(b)\psi(a)\psi(b)\psi(a)\psi(b)\\
=&
\begin{cases}
\psi(b)^{-1}s_{1,2}s_{k+1, k+2}\psi(b) & \text{if $n=2k$}\\
\psi(b)^{-1}s_{1,2}s_{k+1, k+2}s_{k+2,k+3}\psi(b) & \text{if $n=2k+1$}
\end{cases}\\
=&\begin{cases}
 s_{1,2k}s_{k,k+1} & \text{if $n=2k$}\\
 s_{1,2k+1}s_{k,k+1}s_{k+1,k+2} & \text{if $n=2k+1$.}\\
\end{cases}
\end{align*}
 When $n=2k+2$,
$$\psi(b)\psi(a)\psi(b)\psi(a)=j_*(s_{1,2k}s_{k,k+1})\sigma_1^2=i_*(s_{3,2k+2}s_{k+2,k+3})\sigma_1^2=s_{1,2}s_{k+2,k+3}$$
in $B(C_{2k+2})$ and when $n=2k+3$,
\begin{align*}
\psi(b)\psi(a)\psi(b)\psi(a)= &j_*(s_{1,2k+1}s_{k,k+1}s_{k+1, k+2})\sigma_1^2=i_*(s_{3, 2k+3}s_{k+2,k+3}s_{k+3,k+4})\sigma_1^2\\
=&s_{1,2}s_{k+2, k+3}s_{k+3,k+4}
\end{align*}
in $B(C_{2k+3})$, as desired. This completes the proof.
\end{proof}

%%%%%%%%%%%%%%%%%%%%%%%%%%%%%%%%%%%%%%%%%%%%%%%%%%%%%%%%%%%%%%%%%%%%%%

%%%%%%%%%%%%%%%%%%%%%%%%%%%%%%%%%%%%%%%%%%%%%%%%%%%%%%%
%%% Acknowledgements. 致谢
%%%%%%%%%%%%%%%%%%%%%%%%%%%%%%%%%%%%%%%%%%%%%%%%%%%%%%%

%%%%%%%%%%%%%%%%%%%%%%%%%%%%%%%%%%%%%%%%%%%%%%%%%%%%%%%
%%% Conflict of interest. 作者利益声明
%%%%%%%%%%%%%%%%%%%%%%%%%%%%%%%%%%%%%%%%%%%%%%%%%%%%%%%
%\InterestConflict

%%%%%%%%%%%%%%%%%%%%%%%%%%%%%%%%%%%%%%%%%%%%%%%%%%%%%%%
%%% Supplements. 补充材料, 非必选
%%%%%%%%%%%%%%%%%%%%%%%%%%%%%%%%%%%%%%%%%%%%%%%%%%%%%%%
%\Supplements{}

%%%%%%%%%%%%%%%%%%%%%%%%%%%%%%%%%%%%%%%%%%%%%%%%%%%%%%%
%%% Reference section. 参考文献
%%% citation in the content using "some words~\cite{1,2}".
%%% ~ is needed to make the reference number is on the same line with the word before it.
%%%%%%%%%%%%%%%%%%%%%%%%%%%%%%%%%%%%%%%%%%%%%%%%%%%%%%%

\end{document}